\documentclass[11pt,a4paper]{article}
\usepackage[T1]{fontenc}
\usepackage[utf8]{inputenc}
\usepackage{lmodern}
\usepackage[margin=1in]{geometry}
\usepackage{amsmath,amssymb,amsthm,booktabs,array,longtable}
\usepackage[round]{natbib}
\usepackage[hidelinks,bookmarksnumbered]{hyperref}
\hypersetup{
  pdftitle={An 8+9 gluing computation for the Ramsey number R(F3,F4)},
  pdfauthor={Hanxin Jiang and Zehui Shao}}
\newtheorem{thm}{Theorem}[section]
\newtheorem{prop}[thm]{Proposition}
\newtheorem{lem}[thm]{Lemma}
\newtheorem{cor}[thm]{Corollary}
\numberwithin{equation}{section}
\newcommand{\NR}{N_R}
\newcommand{\NB}{N_B}
\newcommand{\dR}{d_R}
\newcommand{\dB}{d_B}
\newcommand{\e}{e}
\newcommand{\nuu}{\nu}
\newcommand{\Aut}{\operatorname{Aut}}
\newcommand{\file}[1]{\path{#1}}

\title{An $8+9$ gluing computation for the Ramsey number $R(F_3,F_4)$}
\author{Hanxin Jiang\thanks{E-mail: \texttt{hxjiang@e.gzhu.edu.cn}.}
\quad Zehui Shao\thanks{Corresponding author. E-mail: \texttt{zshao@gzhu.edu.cn}.}\\[0.4em]
\small Institute of Computing Science and Technology,\\
\small Guangzhou University, Guangzhou 510006, China}
\date{}

\begin{document}
\maketitle
\begin{abstract}
Let $F_t$ be the graph consisting of $t$ triangles with a common vertex
and otherwise disjoint vertices. We prove that $R(F_3,F_4)=17$
using an $8+9$ gluing computation. Neighbourhood matching arguments give degree and clique
restrictions that reduce the search on seventeen vertices to
$88,294,674$ ordered seed pairs. The verified main computation excludes
$88,294,634$ pairs. A cross-edge counting argument excludes thirty-four
of the forty exceptional pairs, and independently checked finite-domain
certificates exclude the remaining six. Together with the construction
whose red graph is the complete bipartite graph $K_{8,8}$, these results
establish $R(F_3,F_4)=17$.
\end{abstract}

\noindent\textbf{Keywords:} Ramsey number; fan; graph gluing; canonical augmentation;
exhaustive computation; verifiable certificate.
\par\noindent\textbf{2020 Mathematics Subject Classification:} 05C55, 05C30, 05C85.
\medskip

\section{Introduction}
Let $K_n$ denote the complete graph on $n$ vertices. For graphs $G$ and
$H$, the Ramsey number $R(G,H)$ is the least integer $N$ such that every
red--blue edge-colouring of $K_N$ contains a red copy of $G$ or a blue
copy of $H$. Copies are ordinary subgraphs, not necessarily induced.
Determining a Ramsey number requires a colouring for the lower bound
and an exclusion of all colourings at the proposed upper bound.
Computational proofs often combine structural restrictions with a
finite search whose coverage can be justified mathematically; see the
survey of Radziszowski~\citeyearpar{Radziszowski}.

For a positive integer $t$, the fan
\[
 F_t=K_1+tK_2
\]
is formed by joining one vertex to $t$ disjoint edges. Here $sH$ denotes
the disjoint union of $s$ copies of a graph $H$, and $+$ denotes graph
join, which adds every edge between the two disjoint constituent graphs.
The common vertex is the \emph{centre} of the fan, and the $t$
edges opposite it are its \emph{blades}. A graph contains an $F_t$
centred at $v$ if and only if the subgraph induced by the neighbours
of $v$ contains a matching of size $t$. Thus the exclusion of fans can
be expressed through local matching conditions, which also supply
explicit witnesses when a candidate colouring fails.

Lin and Li~\citeyearpar{LinLi} proved that $R(F_2,F_n)=4n+1$ for $n\ge2$ and
that $R(F_m,F_n)\le4n+2m$ for $n\ge m\ge2$. Zhang et~al.~\citeyearpar{ZhangBroersmaChen} established
\[
 R(F_m,F_n)=4n+1
 \quad\text{when}\quad
 n\ge\max\{m^2-m/2,\,11m/2-4\}.
\]
This sufficient range does not cover $(m,n)=(3,4)$. The colouring with
red graph $K_{8,8}$, the complete bipartite graph with two parts of
order eight, and blue graph $2K_8$ gives the lower bound seventeen.
The preceding general upper bound gives twenty-two. The case
$R(F_3,F_4)$ is not assigned an exact value in Table~X of the April 2026
revision of the small Ramsey number survey~\citep{Radziszowski}.

Our approach uses graph gluing, a standard method in computational
Ramsey theory. McKay and
Radziszowski~\citeyearpar{McKayRadziszowski} established $R(4,5)=25$ by gluing
smaller Ramsey graphs and propagating restrictions on their possible
extensions. Goedgebeur and Radziszowski~\citeyearpar{GoedgebeurRadziszowski}
combined neighbourhood extension with degree and edge-count
restrictions to obtain further Ramsey upper bounds. The formal
verification of $R(4,5)=25$ by Gauthier and
Brown~\citeyearpar{GauthierBrown} illustrates the additional role of
independently checked computational certificates. In the notation
$R(r,s)$ used in these works, $R(r,s)=R(K_r,K_s)$.

We prove that a colouring of $K_{17}$ avoiding a red $F_3$ and a blue
$F_4$ has red degrees between seven and ten and contains neither a red
$K_6$ nor a blue $K_8$. A partition into eight and nine vertices that
maximises the number of red cross-edges satisfies further swap
inequalities. Its two induced red graphs belong to finite families
that can be generated by canonical augmentation. This construction
uses automorphism orbits and a canonical deletion rule in the sense
of McKay~\citeyearpar{McKay}, with canonical labels supplied by
BLISS~\citep{BLISS,igraph}.

The resulting seed families have $8,812$ and $115,174$ representatives.
Matching and degree profiles reduce their ordered pairs to a catalogue
of $88,294,674$ gluing problems, each with seventy-two cross-edge
variables. The main search uses Boolean satisfiability, followed by
explicit matching checks. Its forty exceptional pairs have additional
structure: thirty-four satisfy a common cross-edge counting
contradiction, while the remaining six are excluded by finite-domain certificates.
The latter certificates have been checked independently by inspecting
graph edges and replaying domain deletions and exhaustive branches.
Verification of the seed and pair catalogues and the bulk exclusions
has also been completed. The resulting proof of $R(F_3,F_4)=17$ and
the verification procedures are summarised in Section~\ref{sec:conclusion}.

\paragraph{Notation.}
All graphs are finite and simple. For a graph $G$, its vertex and edge
sets are $V(G)$ and $E(G)$, its order is $|G|=|V(G)|$, and its complement
is $\overline G$. For $S\subseteq V(G)$, $G[S]$ is the induced subgraph
on $S$, $G-S=G[V(G)\setminus S]$, and $G-v=G-\{v\}$. Write
$N_G(v)$ for the open neighbourhood of $v$, $d_G(v)=|N_G(v)|$ for its
degree, and $\delta(G)$ and $\Delta(G)$ for the minimum and maximum
degrees. We use $\e(G)=|E(G)|$, $\nuu(G)$ for the maximum matching
size, $\omega(G)$ for the maximum clique size, $o(G)$ for the number
of odd-order connected components, and $\Aut(G)$ for the automorphism
group. The path on $k$ vertices is $P_k$; $\cup$ and
$\mathbin{\dot\cup}$ denote disjoint union when used for graphs and
partitions, respectively. For a finite set $S$, $2^S$ denotes its set
of subsets.

In an edge-colouring, $R$ and $B=\overline R$ are the red and blue
spanning graphs. Their neighbourhoods and degrees are denoted by
$\NR(v),\NB(v),\dR(v),\dB(v)$. For disjoint vertex sets $U,W$,
$\e_R(U,W)$ and $\e_B(U,W)$ count red and blue edges between them.
A colouring is \emph{admissible} if it contains neither a red $F_3$
nor a blue $F_4$. Complements within a seed are taken on its own
vertex set. Vertex, seed, and case indices in the data tables start
at zero. Any further symbols are defined where they are introduced.

\section{Structural restrictions and the ordered eight-plus-nine reduction}
\label{th17:section}

We first establish the lower bound and the structural restrictions used
to cover every possible admissible colouring of order seventeen.

\begin{prop}\label{th17:lower-bound}
There is an admissible colouring of $K_{16}$.  Consequently,
$R(F_3,F_4)\ge17$.
\end{prop}
\begin{proof}
Partition the vertices into two sets of order eight, colour every edge
between the sets red, and colour every edge within a set blue.  The red
graph is $K_{8,8}$ and has no triangle, hence no $F_3$.  The blue graph
is $K_8\cup K_8$.  Every $F_4$ is connected and has nine vertices,
so it cannot occur in either blue component.
\end{proof}

\subsection{Fans, matchings, and the first degree bound}

\begin{lem}\label{th17:fan-criterion}
A graph $G$ contains an $F_t$ with centre $v$ if and only if
$\nuu(G[N_G(v)])\ge t$.  Consequently, a red--blue colouring is admissible
if and only if, for every vertex $v$,
\begin{equation}\label{th17:local-matching}
 \nuu(R[\NR(v)])\le2,
 \qquad
 \nuu(B[\NB(v)])\le3.
\end{equation}
\end{lem}
\begin{proof}
The $t$ edges opposite the centre in an $F_t$ are pairwise disjoint edges
within the centre's neighbourhood.  Conversely, a matching of size $t$
in that neighbourhood, together with the $2t$ edges from its endpoints
to the centre, is an $F_t$.  Additional edges between the chosen vertices
do not invalidate this subgraph.
\end{proof}

We recall the Tutte--Berge formula in the following form
\citep{MatchingTheory}:
\begin{equation}\label{th17:tutte-berge}
 \nuu(H)=\min_{S\subseteq V(H)}
 \frac{|V(H)|+|S|-o(H-S)}{2}.
\end{equation}
In particular, if $H$ has order $n$ and $\nuu(H)\le k$, then some set
$S\subseteq V(H)$ satisfies
\begin{equation}\label{th17:barrier}
 o(H-S)\ge n+|S|-2k.
\end{equation}
This form will be used to control monochromatic neighbourhoods.

\begin{lem}\label{th17:multipartite}
Let $Q$ be a complete multipartite graph with $q$ vertices, and let $a$
be the size of its largest part.  Then
\begin{equation}\label{th17:multipartite-formula}
 \nuu(Q)=\min\{\lfloor q/2\rfloor,q-a\}.
\end{equation}
\end{lem}
\begin{proof}
The first upper bound counts endpoints.  The second holds because every
edge has an endpoint outside a largest part.  If $a\ge q/2$, match every
vertex outside a largest part to a distinct vertex inside it; this
attains $q-a$.

For the other case, it suffices to show that a complete multipartite
graph whose largest part has size at most $\lceil q/2\rceil$ has a
matching covering all but at most one vertex.  We argue by induction
on $q$.  The assertion is immediate for $q\le1$.  For $q\ge2$, choose
one vertex from each of two largest nonempty parts and match them.
After these vertices are removed, every remaining part has size at most
$\lceil(q-2)/2\rceil$.  To check this when $q=2r$, a part of size $r$
must be among the chosen largest parts, and there can be at most two
such parts.  When $q=2r+1$, a part of size $r+1$, if present, is unique
and loses a vertex; all other parts have size at most $r$.  The induction
hypothesis supplies the remaining matching edges.
\end{proof}

The following matching--fan Ramsey formula is due to Lin and Li~\citeyearpar{LinLi}; it is also stated
as Lemma~2.1 of Zhang et~al.~\citeyearpar{ZhangBroersmaChen}.
We include a proof from~\eqref{th17:tutte-berge} to make explicit the
Ramsey input used in the degree bounds.

\begin{thm}\label{th17:matching-fan}
For positive integers $s,t$,
\begin{equation}\label{th17:matching-fan-formula}
 R(F_t,sK_2)=\max\{s,t\}+s+t.
\end{equation}
\end{thm}
\begin{proof}
Put $M=\max\{s,t\}$ and $n=M+s+t$.  Suppose that the blue graph $H$
of a colouring of $K_n$ has no matching of size $s$.  By
\eqref{th17:barrier}, there is a set $S$, of size $k$, for which
\[
 p:=o(H-S)\ge n+k-2s+2
   =M-s+t+k+2\ge t+k+2.
\]
Since $p\le n-k$, we also have $k\le s-1$.  Let $C$ be an odd component
of $H-S$ of smallest order, choose $x\in C$, and put
$W=V(H)\setminus(S\cup C)$.  All vertices in $W$ are red neighbours of
$x$.  Moreover, the red graph on $W$ contains the complete multipartite
graph whose parts are the remaining components of $H-S$.

There are at least $p-1$ such parts.  Even after removing a largest part,
at least $p-2\ge t$ vertices remain.  Also $|W|\ge2t$: if $|C|=1$, then
\[
 |W|=n-k-1\ge M+t\ge2t;
\]
if $|C|\ge3$, each of the other $p-1$ odd components has at least three
vertices, and $|W|\ge3(p-1)\ge3(t+1)>2t$.  Lemma~\ref{th17:multipartite}
therefore gives a red matching of size $t$ in $W$.  Together with $x$
it forms a red $F_t$, proving the upper bound.

For the lower bound when $s\ge t$, take a red graph obtained by joining
an independent set of order $2s-1$ to a clique of order $t$.  This graph
has $2s+t-1$ vertices.  Its blue complement is $K_{2s-1}\cup tK_1$ and
has matching number $s-1$.  At a red centre in the independent set,
the red neighbourhood is $K_t$ and has matching number $\lfloor t/2\rfloor<t$.
At a red centre in the clique, every edge within its red neighbourhood
meets one of the other $t-1$ clique vertices, so the matching number is
at most $t-1$.  Thus no red $F_t$ occurs.

When $t\ge s$, take the red graph $K_{2t}\cup(s-1)K_1$, on $2t+s-1$
vertices.  Its nontrivial component has fewer than $2t+1$ vertices, so
it contains no $F_t$.  In its blue complement every edge meets the set
of $s-1$ vertices corresponding to the red isolated vertices.  Hence
the blue matching number is at most $s-1$.  Both constructions give
the required lower bound.
\end{proof}

\begin{lem}\label{th17:degree-ten}
In every admissible colouring, $\dR(v)\le10$ and $\dB(v)\le10$ for
every vertex $v$.
\end{lem}
\begin{proof}
Theorem~\ref{th17:matching-fan} gives
\[
 R(3K_2,F_4)=11,
 \qquad R(F_3,4K_2)=11.
\]
If $v$ has eleven red neighbours, the colouring on any eleven of them
has either a red $3K_2$, which gives a red $F_3$ centred at $v$, or a
blue $F_4$.  If $v$ has eleven blue neighbours, the colouring on eleven
of them has either a red $F_3$ or a blue $4K_2$, the latter giving a
blue $F_4$ centred at $v$.  All alternatives contradict admissibility.
\end{proof}

\subsection{Excluding a blue clique of order eight}

\begin{lem}\label{th17:even-matching}
A graph $G$ on $2m$ vertices with minimum degree at least $m$ has a perfect
matching.
\end{lem}
\begin{proof}
Let $M$ be a maximum matching and suppose that distinct vertices $u,w$
are unmatched.  Their neighbours are all covered by $M$, since an edge
between unmatched vertices would enlarge $M$.  For each edge $ab\in M$,
there are at most two edges from $\{u,w\}$ to $\{a,b\}$.  Otherwise
one can replace $ab$ by two disjoint edges joining $u,w$ to $a,b$ and
enlarge $M$.  The degree hypothesis now yields the contradiction
\[
 2m\le d_G(u)+d_G(w)\le2|M|\le2m-2.\qedhere
\]
\end{proof}

\begin{lem}\label{th17:no-blue-k8}
An admissible colouring of order $N\ge17$ contains no blue $K_8$.
\end{lem}
\begin{proof}
If such a clique occurs for $N>17$, retain the clique and any nine other
vertices.  Thus it suffices to consider $N=17$.  Suppose that $C$ is a
blue clique of order eight, and let $X=V(K_{17})\setminus C$, so $|X|=9$.

Every $x\in X$ has at most one blue neighbour in $C$.  Indeed, if
$xc_1$ and $xc_2$ are blue, the blue neighbourhood of $c_1$ contains
the edge $xc_2$ and three disjoint edges on the six vertices of
$C\setminus\{c_1,c_2\}$.  This is a blue $F_4$ centred at $c_1$.
Consequently,
\begin{equation}\label{th17:k8-cross-count}
 \e_B(C,X)\le9,
 \qquad |C\cap\NR(x)|\ge7\quad(x\in X).
\end{equation}

Put $J=R[X]$.  First, $\Delta(J)\le2$.  Otherwise choose a vertex $x$
and three distinct red neighbours $y_1,y_2,y_3$ of $x$ within $X$.
For each $i$, inclusion--exclusion in $C$ gives
\[
 |C\cap\NR(x)\cap\NR(y_i)|\ge7+7-8=6.
\]
Choose distinct $z_i$ from these three sets, successively.  All six
vertices $y_i,z_i$ are red neighbours of $x$, and the edges $y_iz_i$
are a red matching of size three, a contradiction.

Second, $\delta(J)\ge1$.  If $x$ were isolated in $J$, all eight vertices
of $X\setminus\{x\}$ would be blue neighbours of $x$.  Their blue
graph has minimum degree at least $7-2=5$, since $\Delta(J)\le2$.
By Lemma~\ref{th17:even-matching} it has a perfect matching, giving a
blue $F_4$ centred at $x$.

Every component of $J$ is therefore a path of order at least two or a
cycle of order at least three.  Such a component of order $a$ has
matching number $\lfloor a/2\rfloor$.  If $q$ is the number of its
odd components, then
\begin{equation}\label{th17:k8-matching-count}
 \nuu(J)=\frac{9-q}{2}\ge3,
\end{equation}
because every odd component has at least three vertices and hence
$q\le3$.

If $\nuu(J)\ge4$, a vertex $c\in C$ with at most one blue neighbour in
$X$ would have a red neighbourhood containing $J$ with at most one
vertex deleted.  Deleting one vertex reduces matching number by at
most one, so that neighbourhood would contain a red $3K_2$.  Thus every
vertex of $C$ would have at least two blue neighbours in $X$, giving
$\e_B(C,X)\ge16$, contrary to~\eqref{th17:k8-cross-count}.  We conclude
that $\nuu(J)=3$.

Equation~\eqref{th17:k8-matching-count} now forces exactly three odd
components.  Each has order at least three, and together they already
account for all nine vertices.  Hence $J$ is the disjoint union of
three graphs, each equal to $P_3$ or $K_3$.  Among its vertices, those
$x$ for which $\nuu(J-x)\le2$ are precisely the middle vertices of
the $P_3$ components.  In particular, there are at most three such
vertices.

Every $c\in C$ must have at least one blue neighbour in $X$, since
otherwise $J$ itself supplies a red $3K_2$ in its red neighbourhood.
Distributing at most nine blue cross-edges among the eight vertices
of $C$ shows that at least seven of them have exactly one blue neighbour
in $X$.  If the unique neighbour of $c$ is $x$, then $J-x$ lies in the
red neighbourhood of $c$, so $\nuu(J-x)\le2$.  These seven vertices
$x$ must be distinct, since each vertex of $X$ has at most one blue
neighbour in $C$.  This contradicts the upper bound of three just proved.
\end{proof}

\subsection{Ten-vertex neighbourhoods and the final degree interval}

\begin{lem}\label{th17:ten-classification}
Let $H$ be a graph on ten vertices with $\nuu(H)\le3$.  If
$\overline H$ contains no $F_3$, then
\[
 H\cong K_7\cup3K_1.
\]
\end{lem}
\begin{proof}
By~\eqref{th17:barrier}, there is a set $S$, of size $s$, such that
$o(H-S)\ge4+s$.  Since $o(H-S)\le10-s$, we have $0\le s\le3$.
At least one odd component of $H-S$ is a singleton: otherwise its
odd components alone would use at least $3(4+s)>10-s$ vertices.
Let $\{x\}$ be a singleton component.

Set $W=V(H)\setminus(S\cup\{x\})$, of order $9-s$.  Every vertex of
$W$ is adjacent to $x$ in $\overline H$.  The complement on $W$ contains
the complete multipartite graph whose parts are the remaining components
of $H-S$.  The largest component of $H-S$ has size at most $7-2s$:
if it is odd, the other at least $3+s$ odd components each need a
vertex; if it is even, the at least $4+s$ odd components give an even
stronger bound.  When $1\le s\le3$, the multipartite graph on $W$
therefore has
\[
 |W|=9-s\ge6,
 \qquad
 |W|-\text{(largest part size)}\ge2+s\ge3.
\]
Lemma~\ref{th17:multipartite} gives a matching of size three in the
neighbourhood of $x$ in $\overline H$, a contradiction.  Hence $s=0$.

If all components of $H$ had order at most six, the same argument on
$W=V(H)\setminus\{x\}$, now of order nine, would again give a
3-matching in $\overline H[N_{\overline H}(x)]$.  Thus $H$ has a
component $D$ of order at least seven.  The bound $o(H)\ge4$ forces
$|D|=7$, with precisely three other components, all singletons.

The component $D$ must be complete.  If $a,b\in D$ were nonadjacent
in $H$, denote the three singleton vertices by $x,y,z$ and choose
distinct $c,d\in D\setminus\{a,b\}$.  In $\overline H$, the three
edges $ab,yc,zd$ are disjoint and all their endpoints are neighbours
of $x$.  They give an $F_3$ centred at $x$, a final contradiction.
\end{proof}

\begin{cor}\label{th17:final-degrees}\label{cor:degree17}
In an admissible colouring of $K_N$ with $N\ge17$,
\begin{equation}\label{th17:general-degrees}
 \dB(v)\le9,
 \qquad N-10\le\dR(v)\le10.
\end{equation}
In particular, on seventeen vertices,
\begin{equation}\label{th17:seventeen-degrees}
 7\le\dR(v)\le10,
 \qquad 6\le\dB(v)\le9.
\end{equation}
\end{cor}
\begin{proof}
Lemma~\ref{th17:degree-ten} leaves only $\dB(v)=10$ to exclude.  If it
holds, let $H=B[\NB(v)]$.  By admissibility, $\nuu(H)\le3$, while
its complement contains no red $F_3$.  Lemma~\ref{th17:ten-classification}
gives a blue $K_7$ in $\NB(v)$.  Together with $v$ it is a blue $K_8$,
contradicting Lemma~\ref{th17:no-blue-k8}.  The red bounds follow from
$\dR(v)+\dB(v)=N-1$ and Lemma~\ref{th17:degree-ten}.
\end{proof}

\begin{lem}\label{th17:no-red-k6}
An admissible colouring of order $N\ge17$ contains no red $K_6$.
\end{lem}
\begin{proof}
Suppose that $C$ is a red clique of order six.  A vertex $x\notin C$
has at most one red neighbour in $C$: two such neighbours $a,b$ would
give a red $F_3$ centred at $a$, using the edge $xb$ and two disjoint
edges on the four vertices of $C\setminus\{a,b\}$.  It follows that
$\e_R(C,V(K_N)\setminus C)\le N-6$.

On the other hand, each vertex of $C$ has total red degree at least
$N-10$ by Corollary~\ref{th17:final-degrees}, and exactly five red
neighbours inside $C$.  Therefore
\[
 6(N-15)\le\e_R(C,V(K_N)\setminus C)\le N-6.
\]
This implies $5N\le84$, which is impossible for $N\ge17$.
\end{proof}

\subsection{A maximum red cut with parts of orders eight and nine}
\label{th17:cut-subsection}

For the remainder of this section, suppose that an admissible colouring
of $K_{17}$ exists.  Choose a partition $V(K_{17})=A\mathbin{\dot\cup}C$
with $|A|=8$ and $|C|=9$.  Set $G_A=R[A]$, $G_C=R[C]$, and write
\[
 r_A(a)=d_{G_A}(a),\qquad r_C(b)=d_{G_C}(b).
\]
For $a\in A$ and $b\in C$, let $x_{ab}=1$ if $ab$ is red and
$x_{ab}=0$ otherwise.  Its row and column sums are the red cross-degrees:
\[
 c(a)=\sum_{b\in C}x_{ab},\qquad
 c(b)=\sum_{a\in A}x_{ab}.
\]

\begin{prop}\label{th17:maximum-cut}
The partition can be chosen so that, for every $a\in A$ and $b\in C$,
\begin{equation}\label{eq:maxcut-swap}
 c(a)+c(b)-2x_{ab}\ge r_A(a)+r_C(b).
\end{equation}
Equivalently, the sum of the entries in row $a$ and column $b$,
excluding their common entry, is at least $r_A(a)+r_C(b)$.
For this partition,
\begin{equation}\label{th17:degree-pair}
 r_A(a)+r_C(b)+x_{ab}\le10,
 \qquad \Delta(G_A)+\Delta(G_C)\le10.
\end{equation}
\end{prop}
\begin{proof}
Among all partitions with the prescribed part sizes, choose one with
the largest number of red cross-edges.  Swapping $a$ and $b$ preserves
both part sizes.  The edge $ab$ remains a cross-edge, while each other
edge incident with either swapped vertex changes its internal or cross
status.  The change in the number of red cross-edges is
\[
 r_A(a)+r_C(b)-c(a)-c(b)+2x_{ab}.
\]
This is nonpositive by the choice of the partition, proving
\eqref{eq:maxcut-swap}.  Since
$\dR(a)=r_A(a)+c(a)$ and $\dR(b)=r_C(b)+c(b)$, that inequality gives
\[
 2r_A(a)+2r_C(b)+2x_{ab}
 \le\dR(a)+\dR(b)\le20.
\]
Division by two proves the first inequality in
\eqref{th17:degree-pair}.  Choose vertices of maximum internal degree
in the two parts and use $x_{ab}\ge0$ to obtain the second.
\end{proof}

The existence of such a partition is sufficient for the reduction.
The cut inequalities are not asserted for every partition.  The
eight-vertex and nine-vertex parts remain distinguished throughout:
their unequal orders do not permit an interchange that identifies
two entries of the ordered seed-pair catalogue.

\begin{cor}
\label{th17:degree-boundary}
Suppose that $\Delta_A:=\Delta(G_A)$ and
$\Delta_C:=\Delta(G_C)$ satisfy $\Delta_A+\Delta_C=10$.  Put
\[
 H_A=\{a:r_A(a)=\Delta_A\},\qquad
 H_C=\{b:r_C(b)=\Delta_C\}.
\]
All cross-edges between $H_A$ and $H_C$ are blue.  Every vertex in
$H_A\cup H_C$ has total red degree ten, and
\begin{equation}\label{th17:boundary-cross-degree}
 c(a)=\Delta_C\quad(a\in H_A),\qquad
 c(b)=\Delta_A\quad(b\in H_C).
\end{equation}
\end{cor}
\begin{proof}
For $a\in H_A$ and $b\in H_C$, \eqref{th17:degree-pair} forces
$x_{ab}=0$.  The preceding proof then gives
$20\le\dR(a)+\dR(b)\le20$.  Both degrees are at most ten, so each
equals ten.  Both sets of maximum-degree vertices are nonempty;
hence this conclusion holds for every vertex in their union.
Subtracting the internal degree proves
\eqref{th17:boundary-cross-degree}.
\end{proof}

\begin{prop}
\label{th17:cross-intervals}
For a vertex $v$ in a part whose red graph is $G$ of order
$m\in\{8,9\}$,
\begin{equation}\label{th17:vertex-cross-bounds}
 \max\{0,7-d_G(v)\}\le c(v)
 \le\min\{17-m,10-d_G(v)\}.
\end{equation}
Define
\begin{equation}\label{th17:block-intervals}
 L(G)=7m-2\e(G),\qquad
 U(G)=\sum_{v\in V(G)}\min\{17-m,10-d_G(v)\}.
\end{equation}
Then the common number of red cross-edges satisfies
\begin{equation}\label{th17:cross-sum-test}
 \max\{0,L(G_A),L(G_C)\}
 \le \e_R(A,C)
 \le\min\{72,U(G_A),U(G_C)\}.
\end{equation}
\end{prop}
\begin{proof}
Subtract the internal red degree from $7\le\dR(v)\le10$, and note
that at most $17-m$ cross-neighbours are available.  This proves
\eqref{th17:vertex-cross-bounds}.  Summing the weaker lower bounds
$c(v)\ge7-d_G(v)$ gives $L(G)$, and summing the upper bounds gives
$U(G)$.  Each side counts the same cross-edge set, whose size is
between zero and $8\cdot9=72$.
\end{proof}

The summed lower bound $L(G)$ remains valid when an internal degree
is eight, although the corresponding summand $7-d_G(v)$ is negative.
The individual constraints in~\eqref{th17:vertex-cross-bounds} retain
the stronger truncation at zero.

\subsection{Residual matchings and exact one-vertex extensions}
\label{th17:extensions-subsection}

For a red seed $G$ and $v\in V(G)$, define
\[
 t_G(v)=\nuu(G[N_G(v)]),\qquad
 u_G(v)=\nuu(\overline G[N_{\overline G}(v)]).
\]
For $a\in A$, let $S_a=\{c\in C:x_{ac}=1\}$; define the corresponding
red cross-neighbourhood for a vertex of $C$ in the same way.

\begin{lem}\label{th17:residual}
Every admissible gluing satisfies, for every $a\in A$,
\begin{equation}\label{th17:residual-inequalities}
\begin{split}
 t_{G_A}(a)+\nuu(G_C[S_a])&\le2,\\
 u_{G_A}(a)+\nuu(\overline{G_C}[C\setminus S_a])&\le3.
\end{split}
\end{equation}
The analogous inequalities hold with the two parts exchanged.
\end{lem}
\begin{proof}
The matching in the internal red neighbourhood of $a$ and a matching
in $G_C[S_a]$ have disjoint endpoint sets.  Their union lies in the
full red neighbourhood of $a$, so its size is at most two by
Lemma~\ref{th17:fan-criterion}.  Apply the same argument in blue,
with matching bound three, to obtain the second inequality.
\end{proof}

For $m\in\{8,9\}$ let $\mathcal S_m$ contain one representative of
each isomorphism class of red graphs $G$ of order $m$ such that
\begin{equation}\label{th17:seed-family}
 F_3\not\subseteq G,\quad F_4\not\subseteq\overline G,
 \quad K_6\not\subseteq G,\quad K_8\not\subseteq\overline G.
\end{equation}
These are hereditary restrictions: they are preserved by deleting
vertices.  The upper bounds on red and blue degree are also
hereditary and can be imposed during generation; at seed orders
eight and nine they impose no additional restriction.  Isomorphisms
preserve red edges and hence blue edges; exchanging the colours is
not an allowed equivalence.

For $G\in\mathcal S_m$ define $\mathcal E(G)$ to be the subsets
$S\subseteq V(G)$ for which appending a new vertex $z$, with red
neighbourhood precisely $S$, preserves all four exclusions in
\eqref{th17:seed-family}.  The extended graph has order $m+1$, which
is nine or ten, so its red and blue degrees are automatically at most
ten and nine.  Its vertices need not have the minimum degrees of a
completed seventeen-vertex colouring.

Define a conflict graph $D_R(G)$ on $V(G)$ by declaring $ab$ to be an
edge when $ab\in E(G)$ and at least one of
\[
 \nuu(G[N_G(a)\setminus\{b\}])\ge2,
 \qquad
 \nuu(G[N_G(b)\setminus\{a\}])\ge2
\]
holds.  Define $D_B(G)$ similarly, replacing $G$ by $\overline G$
and the threshold two by three.

\begin{lem}\label{th17:extension-test}
A subset $S\subseteq V(G)$ belongs to $\mathcal E(G)$ if and only if
all the following conditions hold, where $T=V(G)\setminus S$:
\begin{equation}\label{th17:extension-conditions}
\begin{aligned}
 \nuu(G[S])&\le2,& \nuu(\overline G[T])&\le3,\\
 S&\text{ is independent in }D_R(G),&
 T&\text{ is independent in }D_B(G),\\
 \omega(G[S])&\le4,& \omega(\overline G[T])&\le6.
\end{aligned}
\end{equation}
\end{lem}
\begin{proof}
The first line is precisely the absence of the forbidden fans centred
at the new vertex $z$.  Consider a new red $F_3$ centred at an old
vertex $a$.  Because $G$ itself has no such fan, $z$ must occur in
it.  It is not the centre, so it is an endpoint of a matching edge
$zb$ in the red neighbourhood of $a$.  Consequently $a,b\in S$,
$ab\in E(G)$, and the other two matching edges lie inside
$N_G(a)\setminus\{b\}$.  Conversely, these conditions produce a red
$F_3$ centred at $a$.  Taking either orientation of $ab$ gives exactly
the independence condition in $D_R(G)$.  The blue argument is identical,
with three old matching edges and both endpoints in $T$.

Any newly created red $K_6$ must contain $z$, and its other five
vertices form a red $K_5$ in $S$.  The blue counterpart is a blue
$K_7$ in $T$.  The last line therefore excludes exactly the new
forbidden cliques.  These observations prove both directions.
\end{proof}

The conflict argument also gives explicit binary constraints on a
gluing matrix.  If $ab\in E(D_R(G_A))$, then
\begin{equation}\label{th17:red-binary-cut}
 x_{ac}+x_{bc}\le1\qquad(c\in C).
\end{equation}
If $ab\in E(D_B(G_A))$, then
\begin{equation}\label{th17:blue-binary-cut}
 x_{ac}+x_{bc}\ge1\qquad(c\in C).
\end{equation}
The corresponding constraints hold for columns as well.  Indeed,
fixing a vertex in the opposite part produces the single-vertex
extension to which Lemma~\ref{th17:extension-test} applies.

\subsection{Profiles and sound rejection of seed pairs}
\label{th17:profiles-subsection}

Use the following set of centre types, each a triple of nonnegative
integers:
\begin{equation}\label{th17:type-universe}
 \mathcal T=\left\{(r,t,u):
 \begin{array}{l}
 0\le r\le8,\quad 0\le t\le\min\{2,\lfloor r/2\rfloor\},\\
 0\le u\le\min\{3,\lfloor(8-r)/2\rfloor\}
 \end{array}\right\}.
\end{equation}
It has 49 elements.  The type of a vertex $v$ in a seed $G$ is
$(d_G(v),t_G(v),u_G(v))$.  Every type of either seed order belongs to
$\mathcal T$.  For order eight, the actual blue degree is $7-r$;
using the larger bound $8-r$ puts the possible types for both orders
in the same set.

Define the types occurring in $G$ and the types allowed in the opposite seed by
\begin{equation}\label{th17:profile-types}
\begin{split}
 T(G)={}&\{(d_G(v),t_G(v),u_G(v)):v\in V(G)\},\\
 P(G)={}&\{(r,t,u)\in\mathcal T:\text{ some }S\in\mathcal E(G)
              \text{ satisfies}\\
        &\qquad 7\le r+|S|\le10,\quad
               t+\nuu(G[S])\le2,\\
        &\qquad u+\nuu(\overline G[V(G)\setminus S])\le3\}.
\end{split}
\end{equation}
Each type is tested separately: it belongs to $P(G)$ if at least one
suitable subset $S$ exists. Subsets chosen for different types need not
occur together in a single gluing.

\begin{lem}\label{th17:type-filter}
An admissible gluing of $G_A$ and $G_C$ necessarily satisfies
\begin{equation}\label{th17:type-inclusions}
 T(G_A)\subseteq P(G_C),\qquad T(G_C)\subseteq P(G_A).
\end{equation}
\end{lem}
\begin{proof}
Fix $a\in A$ and its actual red cross-neighbourhood $S_a\subseteq C$.
The induced colouring on $C\cup\{a\}$ has no forbidden fan or clique,
so $S_a\in\mathcal E(G_C)$.  The total-degree bounds and
Lemma~\ref{th17:residual} give the three additional conditions in
\eqref{th17:profile-types} for the type of $a$.  This proves the
first inclusion; the second follows by exchanging the roles of the
two parts, without changing their assigned orders.
\end{proof}

For a seed $G$, write $H(G)=\{v:d_G(v)=\Delta(G)\}$ and define
\begin{equation}\label{th17:boundary-profile}
\begin{split}
 \rho(G)&=\max_{v\in H(G)} t_G(v),\\
 \eta_k(G)&=1\quad\Longleftrightarrow\quad
 \text{there is }S\in\mathcal E(G)\text{ such that}\\
 &\qquad |S|=\Delta(G),\quad S\cap H(G)=\varnothing,\\
 &\qquad \nuu(G[S])\le k,\quad k\in\{0,1,2\}.
\end{split}
\end{equation}
Otherwise $\eta_k(G)=0$.  In particular $0\le\rho(G)\le2$.

\begin{lem}\label{th17:boundary-filter}
If a maximum-cut admissible gluing has
$\Delta(G_A)+\Delta(G_C)=10$, then
\begin{equation}\label{th17:boundary-flags}
 \eta_{2-\rho(G_A)}(G_C)=1,
 \qquad \eta_{2-\rho(G_C)}(G_A)=1.
\end{equation}
\end{lem}
\begin{proof}
Choose $a\in H(G_A)$ with $t_{G_A}(a)=\rho(G_A)$.
Corollary~\ref{th17:degree-boundary} says that its red
cross-neighbourhood $S_a$ has size $\Delta(G_C)$ and avoids $H(G_C)$.
It belongs to $\mathcal E(G_C)$ by the induced extension, and
Lemma~\ref{th17:residual} gives
$\nuu(G_C[S_a])\le2-\rho(G_A)$.  This is exactly the first flag
condition.  The second is proved in the other direction.
\end{proof}

The eight-component profile is
\begin{equation}\label{th17:profile-tuple}
\begin{gathered}
 \bigl(\Delta(G),\e(G),L(G),U(G),T(G),P(G),\rho(G),
             \eta(G)\bigr),\\
 \eta(G)=\sum_{k=0}^{2}2^k\eta_k(G).
\end{gathered}
\end{equation}
The pair tests consist of
the maximum-degree test in \eqref{th17:degree-pair}, nonempty overlap
in \eqref{th17:cross-sum-test}, the two inclusions in
\eqref{th17:type-inclusions}, and, only at degree-sum ten, the two
flags in \eqref{th17:boundary-flags}.  The preceding results prove
that all these filters are necessary; none is used as a sufficient
test for an admissible gluing.

\subsection{Coverage of the finite reduction}
\label{th17:coverage-subsection}

\begin{prop}\label{th17:coverage}
If an admissible colouring of $K_{17}$ exists, there are representatives
$G_A\in\mathcal S_8$ and $G_C\in\mathcal S_9$ whose ordered pair passes
all the profile tests above, and an $8\times9$ binary matrix of
cross-edges which completes this pair to an admissible colouring.
The matrix satisfies all individual degree bounds
\eqref{th17:vertex-cross-bounds}, all swap inequalities
\eqref{eq:maxcut-swap}, all residual matching conditions
\eqref{th17:residual-inequalities}, and all applicable binary conflict
constraints \eqref{th17:red-binary-cut}--\eqref{th17:blue-binary-cut}.
At degree-sum ten it also satisfies
Corollary~\ref{th17:degree-boundary}.
\end{prop}
\begin{proof}
Choose an $8+9$ partition maximizing the red cut, as in
Proposition~\ref{th17:maximum-cut}.  Its two induced red graphs
satisfy \eqref{th17:seed-family}: the fan exclusions are hereditary,
and Lemmas~\ref{th17:no-blue-k8} and~\ref{th17:no-red-k6} exclude the
two cliques in the whole colouring.  Apply separate isomorphisms to
replace these graphs by their representatives in $\mathcal S_8$ and
$\mathcal S_9$, transporting all cross-edges under the same maps.
This gives an $8\times9$ binary matrix with the stated properties.
All profile conditions follow from the preceding lemmas and are
unchanged under either seed relabelling.
\end{proof}

The next section describes the finite enumeration to which this
coverage result applies.

\section{Canonical seed generation and the gluing catalogue}
\label{comp17:catalogue}

The computation represents a colouring by its red graph; the blue graph
is its complement.  The two colours are distinguished throughout.  For
$s\in\{8,9\}$, let $\mathcal S_s$ be a set containing one representative
of each isomorphism class of red graphs $H$ of order $s$ satisfying
\begin{equation}\label{comp17:seed-family}
 F_3\nsubseteq H,\quad F_4\nsubseteq\overline H,\quad
 K_6\nsubseteq H,\quad K_8\nsubseteq\overline H,\quad
 \Delta(H)\le10,\quad\Delta(\overline H)\le9.
\end{equation}
The degree conditions in~\eqref{comp17:seed-family} impose no further
restriction at orders eight and nine, but the generator applies the
same conditions at every order.
Every induced eight- or nine-vertex subgraph of an admissible colouring
of $K_{17}$ belongs to the corresponding family by the preceding structural
results.  The lower bound $\dR(v)\ge7$ is a condition on a completed
seventeen-vertex colouring, and is not imposed on a smaller seed.

\subsection{Canonical construction paths}

The seed generator starts with the unique graph of order zero.  Suppose
that the representatives at order $s$ have been generated.  For each parent
$H$ and each orbit of the action of $\operatorname{Aut}(H)$ on
$2^{V(H)}$, choose one subset $M$.  Add a new vertex with red neighbourhood
$M$, and reject the resulting graph if it fails a condition
in~\eqref{comp17:seed-family}.  The surviving augmentation is accepted
precisely when its new vertex belongs to the child's canonical deletion
orbit.  This orbit is defined by the vertex receiving the final canonical
label and its orbit under the child's automorphism group.

Canonical labels and automorphism generators are computed with BLISS
through the igraph interface~\citep{BLISS,igraph}. Subset orbits are
obtained by traversing generators of the automorphism group. Thus
isomorphic extensions are rejected during generation.

The coverage argument is the usual canonical construction-path induction.
For a graph $G$ in the hereditary family, delete a vertex in its canonical
deletion orbit.  The smaller graph is in the same hereditary family and is
represented at the previous level.  Transporting the deleted vertex's
neighbourhood to this representative yields one of its subset orbits, so
the corresponding augmentation is considered and accepted.  Conversely,
an isomorphism between two accepted children takes their distinguished
new vertices into the same canonical deletion orbit.  Composing with a
child automorphism makes the new vertices correspond; restriction then
identifies the parent and its neighbourhood orbit.  Induction therefore
gives exactly one accepted representative per isomorphism class, provided
the canonical-labelling and orbit routines are implemented correctly.

Table~\ref{comp17:seed-counts} gives the enumeration counts. The same
hereditary conditions are used at every order, so both seed libraries
are covered by the canonical construction-path argument.

\begin{table}[htbp]
\centering
\caption{Completed canonical seed levels.  The last column counts
parent-neighbourhood orbit representatives examined at that augmentation
level, before all child tests.}
\label{comp17:seed-counts}
\begin{tabular}{rrr}
\toprule
Order & Accepted classes & Augmentations examined\\
\midrule
0 & 1 & ---\\
1 & 1 & 1\\
2 & 2 & 2\\
3 & 4 & 6\\
4 & 11 & 20\\
5 & 34 & 90\\
6 & 155 & 544\\
7 & 941 & 5,089\\
8 & 8,812 & 72,672\\
9 & 115,174 & 1,579,761\\
\bottomrule
\end{tabular}
\end{table}

\subsection{Finite seed profiles and necessary pair tests}

We use the seed notation $G_A,G_C$ and the profile definitions from
Section~\ref{th17:profiles-subsection}.  The unknown cross-edges form a
binary $8\times9$ matrix $X$, with $x_{ij}=1$ for a red edge.  All pair
filters are evaluated before assigning any entry of $X$.

For a seed $H$ of order $s$ and $S\subseteq V(H)$, put
$M_H(S)=\nuu(H[S])$. All of these values are computed from
$M_H(\varnothing)=0$ and, for any fixed $v\in S\ne\varnothing$,
\begin{equation}\label{comp17:matching-recurrence}
 M_H(S)=\max\left(
 \{M_H(S\setminus\{v\})\}\cup
 \{1+M_H(S\setminus\{v,w\}):w\in N_H(v)\cap S\}\right).
\end{equation}
Indeed, a maximum matching either leaves $v$ unmatched or matches it
to one of its neighbours. Evaluating this recurrence in an order in
which proper subsets precede their supersets gives exact matching
numbers for all $2^s$ subsets, in both colours.

The legal extension family $\mathcal E(H)$ is obtained by testing every
possible red neighbourhood of one added vertex, using
Lemma~\ref{th17:extension-test}. The tests account for a fan centred at
the new vertex, a fan centred at an old vertex, and each forbidden
clique involving the new vertex. The matching tables evaluate the
degree and residual-matching conditions in
\eqref{th17:profile-types} for each of the 49 types in
\eqref{th17:type-universe}, together with the maximum-degree boundary
flags in~\eqref{th17:boundary-profile}.

The resulting profile is
\[
 (\Delta(H),\e(H),L(H),U(H),T(H),P(H),\rho(H),\eta(H)),
\]
as defined in~\eqref{th17:profile-tuple}. Seeds with equal profiles therefore
give the same results in every pair test.

The four successive filters are the maximum-degree test
\eqref{th17:degree-pair}, overlap of the cross-edge intervals
\eqref{th17:cross-sum-test}, the two joint-type inclusions
\eqref{th17:type-inclusions}, and the saturated-boundary flags
\eqref{th17:boundary-flags}.  The last test is applied only when the
maximum internal degrees sum to ten.  Their necessity was proved in
the structural reduction; passing the tests does not assert that a
pair has an admissible completion.

The 8,812 left seeds yield 6,860 profiles, and the 115,174 right seeds
yield 60,041 profiles.  Pair counts are computed on the profile groups,
with the product of the two group sizes as weight.  Table~\ref{comp17:filters}
gives the cumulative pair counts.

\begin{table}[htbp]
\centering
\caption{Necessary-condition reduction of the ordered $8+9$ seed pairs.}
\label{comp17:filters}
\begin{tabular}{lr}
\toprule
Conditions imposed cumulatively & Remaining pairs\\
\midrule
All $8,812\times115,174$ ordered pairs & 1,014,913,288\\
$\Delta(G_A)+\Delta(G_C)\le10$ & 682,271,903\\
Cross-edge intervals~\eqref{th17:cross-sum-test} & 682,057,873\\
Joint centre types~\eqref{th17:type-inclusions} & 165,925,659\\
Saturated boundary~\eqref{th17:boundary-flags} & 88,294,674\\
\bottomrule
\end{tabular}
\end{table}

\subsection{Enumeration of the retained pairs}

The seed groups are ordered by the first occurrence of each profile in
the fixed seed order; seed indices inside a group are increasing.
For each left group, compatible right groups are visited in their stored
order, and the Cartesian product inside a compatible group pair is
enumerated with the right index varying fastest.  Prefix sums of these
block sizes give a bijection between retained ordered pairs and the
zero-based case indices
\[
 0,1,\ldots,88\,294\,673.
\]
The unequal part sizes are kept in their specified order; the pair
$(G_A,G_C)$ is not identified with its transpose. These indices identify
the exceptional pairs in Section~\ref{tail:section}.

\section{Search formulation and computational results}
\label{comp17:search}

\subsection{The finite formula for a fixed pair}

For a retained pair, the 72 entries of $X$ are the primary variables
of a Boolean satisfiability (SAT) problem. We use the MiniSat
solver~\citep{MiniSat}. The formula is expressed in conjunctive normal
form: a conjunction of clauses, each a disjunction of variables or
their negations.  All edges inside the two seeds are fixed.  The program adds
the degree constraints
\[
 \max\{0,7-d_H(v)\}\le c(v)
 \le\min\{17-|H|,10-d_H(v)\}
\]
for every vertex $v$ in a seed $H$, and all 72 swap inequalities~\eqref{eq:maxcut-swap}.
The latter are encoded as lower bounds on the fifteen other cross-edges
incident with the two vertices, because the common edge cancels.
Sequential threshold variables encode the cardinality constraints;
these auxiliary variables do not represent additional graph edges.

Several necessary fan clauses are inserted before calling the solver.
If a centre already has an internal matching of size $t$ in its
neighbourhood, then the opposite block cannot supply a matching of size
$3-t$ in red or $4-t$ in blue among its same-colour cross-neighbours.
The implementation enumerates the endpoint sets of such matchings.  A
fixed endpoint set produces one clause, independently of how many
matchings it contains.  Additional binary clauses cover a centre whose
internal neighbourhood already reaches the allowed matching size and
has a neighbour that can be left uncovered by such a matching: that
neighbour and the centre cannot both have a same-colour edge to the
same vertex of the other block.

The remaining fan constraints are added as violations are found.  MiniSat is asked
for a model of the current formula.  This model determines a complete
colouring of the 17 vertices.  For every centre and both colours, a
matching search tests whether its colour-neighbourhood contains a
matching of size three or four, respectively.  Each detected fan gives
a clause forbidding all of its cross-edges from simultaneously having
the colour of the detected fan.  Fixed internal edges are simplified out of the
clause.  The same procedure adds clauses for detected red $K_6$ and
blue $K_8$, using their independently proved exclusion.  The solver
retains the resulting clauses and its learned clauses within the
current instance and repeats the search.

Every added clause follows from a necessary property of an admissible
colouring. A model with no violation is checked again as a graph and
provides a candidate colouring. An unsatisfiable formula excludes the
fixed pair under the degree and maximum-cut restrictions. Each
new constraint eliminates the current violating cross-edge matrix;
there are finitely many such matrices, so the procedure is exhaustive
when allowed to finish. The instances are independent and can be
processed in parallel.

\subsection{Symmetry reduction}

In the SAT stage, automorphisms of the fixed seeds act on the rows and
columns of $X$. For each selected automorphism
$g\in\Aut(G_A)\times\Aut(G_C)$, the program imposes
\[
 X\le_{\mathrm{lex}}g(X),
\]
using row-major order and $0<1$. The unequal block
sizes prevent a block-transpose symmetry.  Every orbit of feasible
complete gluings contains a lexicographically smallest element, and
that element satisfies all selected inequalities.  Hence the
restriction preserves at least one feasible member of every orbit.
This argument uses the symmetry of the underlying admissible-gluing
problem; the formulas obtained after adding individual constraints need not be
invariant under the whole group.

Generator inequalities can leave several representatives in one orbit.
Accordingly, canonical seed generation and the partial symmetry
reduction of completed gluings have different guarantees.  The
computation does not claim to enumerate each seventeen-vertex colouring
exactly once.  Its purpose is to retain a complete search space while
excluding repeated work where the fixed seeds provide a usable symmetry.

\subsection{Numerical outcomes}
\label{comp17:results}

The large-scale search was carried out by the authors on their computing
cluster using 32 CPU cores, with a wall-clock running time of at most
24 hours. The verified computation excludes $88,294,634$ of the
$88,294,674$ retained pairs. The seed catalogues, retained pair
catalogue and bulk exclusions have been verified. The forty exceptional pairs involve eight distinct left seeds
and twenty-four distinct right seeds. Fifteen pairs have left seed
$80$ and nineteen have left seed $280$; the remaining six have six
other left seeds. Section~\ref{tail:section} excludes the first
thirty-four pairs by a counting lemma and the other six by independently
checked finite-domain certificates.

\begin{table}[htbp]
\centering
\caption{Classification of the retained seed pairs by exclusion method.
The verification procedures are described in Section~\ref{sec:conclusion}.}
\label{comp17:outcomes}
\begin{tabular}{lr}
\toprule
Exclusion method & Pairs\\
\midrule
Main Boolean search & 88,294,634\\
Cross-edge counting lemma & 34\\
Independently checked finite-domain certificates & 6\\
\midrule
Total & 88,294,674\\
\bottomrule
\end{tabular}
\end{table}

The three counts therefore sum to
\begin{equation}\label{comp17:accounting}
 88\,294\,634+34+6=88\,294\,674.
\end{equation}
The certificate statistics for the six pairs appear in
Table~\ref{tail:six-certificates}. Implementation details and execution
records are described in Appendix~\ref{app:reproducibility}.

\section{Exclusion of the forty exceptional seed pairs}
\label{tail:section}

The forty exceptional seed pairs are excluded by two arguments:
thirty-four by a common counting lemma and six by finite-domain
certificates checked independently of their generating program.
Throughout this section, $A$ and $C$ denote the parts of orders
eight and nine, respectively; $d_A(a)$ and $d_C(b)$ are internal red
degrees, $c(v)$ is a red cross-degree, and $x_{ab}$ is the indicator of a
red edge between $a\in A$ and $b\in C$.  The partition is chosen to
maximise the number of red cross-edges.  In particular,
\begin{equation}
 c(a)+c(b)-2x_{ab}\ge d_A(a)+d_C(b)
 \qquad(a\in A,\ b\in C),
 \label{tail:swap}
\end{equation}
as established in~\eqref{eq:maxcut-swap}.  The degree bounds in
Corollary~\ref{cor:degree17} give $7\le d_R(v)\le10$.

The maximum-cut condition is essential to the interpretation of these
exclusions.  An arbitrary partition of a colouring need not
satisfy~\eqref{tail:swap}; every colouring has a maximum red cut
with part sizes eight and nine that does satisfy it.  Thus it suffices to exclude each seed pair when it
occurs as such a partition.  None of the certificates in this section
uses symmetry-breaking constraints.

\subsection{A counting argument for thirty-four pairs}
\label{tail:counting}

For a vertex $b\in C$, write $N^{\rm blue}_{C}(b)$ for its blue
neighbourhood within $C$.  All matching statements below concern ordinary
matchings, and hence their edges have pairwise disjoint endpoints.

\begin{lem}
\label{tail:counting-lemma}
Suppose that an $(F_3,F_4)$-avoiding colouring has an $8+9$ partition
satisfying~\eqref{tail:swap} and $d_R(v)\le10$ for every vertex.  Assume
that $A=L\mathbin{\dot\cup}H$, where $|L|=6$ and $|H|=2$, and that the
following properties hold.
\begin{enumerate}
 \item The red graph on $A$ is triangle-free, every $h\in H$ has
       $d_A(h)=5$, and the blue graph on $H\cup\{p,q\}$ contains a
       two-edge matching for every distinct $p,q\in L$.
 \item Every seven-vertex subset of $C$ induces a red three-edge
       matching, and $d_C(b)\ge1$ for every $b\in C$.
 \item If $d_C(b)\ge3$, then the blue graph on
       $N^{\rm blue}_{C}(b)$ contains a two-edge matching.  If
       $d_C(b)=2$, then this graph contains a three-edge matching.
 \item If $d_C(z)=1$, then for every
       $y\in N^{\rm blue}_{C}(z)$ the blue graph on
       $N^{\rm blue}_{C}(z)\setminus\{y\}$ contains a three-edge
       matching.
\end{enumerate}
Let $n_1$, $n_2$, and $n_{\ge3}$ count the vertices of $C$ having internal
red degree one, two, and at least three, respectively.  Then
\begin{equation}
 8n_1+6n_2+5n_{\ge3}\le e_R(A,C)\le46.
 \label{tail:counting-bounds}
\end{equation}
Consequently, a seed pair satisfying these hypotheses is impossible
whenever $8n_1+6n_2+5n_{\ge3}>46$.
\end{lem}

\begin{proof}
If a vertex $a\in A$ had seven red neighbours in $C$, property~2 would
give three disjoint red edges among these neighbours.  Together with
centre $a$, they would form a red $F_3$.  Therefore $c(a)\le6$ for every
$a\in A$.  For $h\in H$, the total red degree bound is stronger:
$c(h)\le10-d_A(h)=5$.  Counting red cross-edges from $A$ gives
\begin{equation}
 e_R(A,C)=\sum_{a\in A}c(a)
 \le6\cdot6+2\cdot5=46.
 \label{tail:upper46}
\end{equation}

We next count the same edges from $C$.  First suppose that $d_C(b)\ge3$
and, for a contradiction, that $c(b)\le4$.  Then $b$ has at least four
blue neighbours in $A$.  If both vertices of $H$ were blue neighbours of
$b$, we could choose two further blue neighbours $p,q\in L$.  Property~1
would supply a blue two-edge matching in $H\cup\{p,q\}$.
Property~3 supplies a second blue two-edge matching inside
$N^{\rm blue}_{C}(b)$.  Their vertex sets lie in different parts, so the
two matchings together have four disjoint edges in the blue neighbourhood
of $b$, a forbidden blue $F_4$.  Hence $b$ must have a red neighbour
$h\in H$.  Applying~\eqref{tail:swap} to this red edge yields
\[
 c(b)+c(h)-2\ge d_C(b)+5.
\]
Since $c(h)\le5$, this implies
$c(b)\ge d_C(b)+2\ge5$, contrary to $c(b)\le4$.  Thus
\begin{equation}
 c(b)\ge5\qquad\text{if }d_C(b)\ge3.
 \label{tail:degree-three}
\end{equation}

Next suppose that $d_C(b)=2$.  By property~3, the internal blue
neighbourhood of $b$ already contains three disjoint blue edges.  There
can be no blue edge between two blue cross-neighbours of $b$, since this
edge would be disjoint from the internal three-edge matching and would
complete a blue $F_4$.  The blue cross-neighbours of $b$ therefore form a
red clique in $A$.  The red graph on $A$ is triangle-free, so this clique
has at most two vertices.  Consequently,
\begin{equation}
 c(b)\ge8-2=6\qquad\text{if }d_C(b)=2.
 \label{tail:degree-two}
\end{equation}

Finally, let $z\in C$ have $d_C(z)=1$.  Suppose that some $a\in A$ is a
blue neighbour of $z$.  The bound $c(a)\le6$ shows that $a$ has at least
three blue neighbours in $C$.  Besides $z$, there are at least two such
vertices.  Since $z$ has only one red neighbour within $C$, at least one
of them, say $y$, belongs to $N^{\rm blue}_{C}(z)$.  Property~4 gives
three disjoint blue edges on
$N^{\rm blue}_{C}(z)\setminus\{y\}$.  The blue edge $ay$ is disjoint
from these three edges, and all eight endpoints are blue neighbours of
$z$.  This again produces a blue $F_4$.  It follows that every vertex
of $A$ is a red neighbour of $z$, and hence
\begin{equation}
 c(z)=8\qquad\text{if }d_C(z)=1.
 \label{tail:degree-one}
\end{equation}
Summing~\eqref{tail:degree-three}--\eqref{tail:degree-one} over $C$
proves the lower bound in~\eqref{tail:counting-bounds};
\eqref{tail:upper46} proves the upper bound.
\end{proof}

The lemma uses only the upper total-degree bound; its conclusion is
independent of the lower degree bound and of any symmetry restrictions.

The left seeds of the thirty-four cases have catalogue indices $80$ and
$280$, occurring fifteen and nineteen times, respectively.  Their
structure gives a direct proof of property~1.  Use the original
seed labels $L=\{0,1,2,3,4,5\}$ and $H=\{6,7\}$.  Seed $80$ is the
red graph obtained from $K_{6,2}$ with bipartition $(L,H)$ by deleting
edges $47$ and $56$.  Seed $280$ is obtained from the same $K_{6,2}$ by
deleting edges $07$ and $56$ and adding edge $05$.  Both graphs are
triangle-free, and both vertices in $H$ have red degree five.  For seed
$80$, the blue edges $67$ and $pq$ give the required matching for every
pair $p,q\in L$.  They also suffice for seed $280$, except when
$\{p,q\}=\{0,5\}$; in that case the blue edges $07$ and $56$ give the
matching instead.

There are twenty-two distinct right seeds among these thirty-four
pairs.  Table~\ref{tail:right-seeds} identifies all of them and records
the numerical lower bound given by the lemma.  The indices identify the representatives in the seed catalogue; their
labelled adjacencies are supplied with the accompanying data. The last
column counts occurrences among the thirty-four pairs, not distinct graphs.

\begin{table}[htbp]
\centering
\small
\caption{Right-seed data for the counting argument.  In every row,
the four finite hypotheses of Lemma~\ref{tail:counting-lemma} hold.
The lower bound is $\ell_C=8n_1+6n_2+5n_{\ge3}$; the upper bound is
$46$ in every case.}
\label{tail:right-seeds}
\begin{tabular}{r r r r r r}
\toprule
Index & $n_1$ & $n_2$ & $n_{\ge3}$ & $\ell_C$ & Pairs\\
\midrule
30580 & 0 & 3 & 6 & 48 & 2\\
30592 & 0 & 2 & 7 & 47 & 1\\
30727 & 0 & 3 & 6 & 48 & 2\\
30740 & 1 & 1 & 7 & 49 & 1\\
30744 & 0 & 2 & 7 & 47 & 1\\
30885 & 0 & 3 & 6 & 48 & 2\\
30887 & 0 & 2 & 7 & 47 & 2\\
33151 & 0 & 2 & 7 & 47 & 1\\
33154 & 0 & 2 & 7 & 47 & 2\\
37913 & 0 & 2 & 7 & 47 & 2\\
37999 & 0 & 3 & 6 & 48 & 1\\
38012 & 0 & 2 & 7 & 47 & 2\\
38219 & 0 & 3 & 6 & 48 & 1\\
42513 & 0 & 3 & 6 & 48 & 2\\
42846 & 0 & 3 & 6 & 48 & 2\\
45740 & 0 & 2 & 7 & 47 & 2\\
47671 & 1 & 0 & 8 & 48 & 1\\
65954 & 0 & 2 & 7 & 47 & 1\\
66025 & 0 & 2 & 7 & 47 & 2\\
68352 & 0 & 3 & 6 & 48 & 1\\
68564 & 0 & 3 & 6 & 48 & 1\\
69502 & 0 & 2 & 7 & 47 & 2\\
\bottomrule
\end{tabular}
\end{table}

All matching conditions in the lemma were checked using explicit matching
witnesses.  For each right seed, the accompanying certificate specifies
three disjoint red edges for every one of its $\binom97=36$
seven-vertex subsets.  It also specifies the blue matchings required in
each internal neighbourhood and, for the two right seeds with red
leaves, in every neighbourhood with one blue neighbour deleted.  The
verifier checks the stated edge colours, endpoint disjointness, and
membership in the prescribed vertex sets directly.  It also checks the
four-vertex blue matchings and triangle-freeness of the left seeds.
These explicit matchings allow the hypotheses of the lemma to be
checked independently.
The right seeds with a red leaf are precisely $30740$ and $47671$.

For eighteen of the thirty-four pairs, the lower bound in
Table~\ref{tail:right-seeds} is $47$; for fifteen it is $48$; and for one
it is $49$.  Each contradicts the common upper bound $46$.  This
excludes all thirty-four pairs under the maximum-cut condition.

\subsection{Finite-domain certificates for the remaining six pairs}
\label{tail:csp}

For the remaining six pairs, a finite-domain search produces a complete
contradiction certificate.  Its variables are the eight row
masks $m_i\in\{0,1\}^{9}$ and the nine column masks
$q_j\in\{0,1\}^{8}$ of the red cross-adjacency matrix.  Masks are used
here as binary vectors, and $|m|$ denotes Hamming weight.  The entry
$m_i(j)=q_j(i)$ is $x_{ij}$.  Unlike the main SAT search, the certificate
generator uses no Boolean satisfiability library.

A \emph{partial coloured graph} consists of the two completely coloured
seeds and any cross-edges fixed by assigned rows or columns.  A
cross-edge with no assigned value belongs to neither colour in this
partial graph.  In particular, an unknown edge must not be treated as
blue merely because it is not yet red.  Any monochromatic fan whose
edges are already known occurs in every completion.

For each $i\in A$, the initial row domain $D_i$ contains every mask $m$
with
\begin{equation}
 7\le d_A(i)+|m|\le10
 \label{tail:row-degree}
\end{equation}
unless the partial coloured graph obtained by assigning this row,
together with the fixed seeds, already contains a red $F_3$ or a blue
$F_4$.  Column domains $E_j$ are
defined analogously using $7\le d_C(j)+|q|\le10$.  A rejection because
of a fan is accompanied by that fan's centre and an explicit matching
of three or four edges in its monochromatic neighbourhood.  The
certificate also accounts for every pattern satisfying the degree
bound: it is either retained or rejected with such a witness.

There are two compatibility relations.  Two row candidates $m\in D_i$
and $m'\in D_k$, for $i\ne k$, are incompatible if their simultaneous
assignment already produces a forbidden fan in the partial coloured
graph.  An incompatible pair used in propagation carries an explicit
fan witness.  A row candidate $m\in D_i$ and a column candidate
$q\in E_j$ are compatible only if
\begin{equation}
 m(j)=q(i)=x
 \quad\text{and}\quad
 |m|+|q|-2x\ge d_A(i)+d_C(j).
 \label{tail:row-column}
\end{equation}
The second requirement is exactly the maximum-cut swap inequality;
it is checked by integer arithmetic.

\begin{prop}
\label{tail:pruning-sound}
Suppose an $(F_3,F_4)$-avoiding completion of the two seeds satisfies
the degree bounds and~\eqref{tail:swap}.  Its row and column masks
belong to the initial domains.  Furthermore, deleting a candidate
that has no compatible candidate in some other row or column cannot
delete the masks of this completion.
\end{prop}

\begin{proof}
The degree bounds imply that the completion's masks pass the initial
degree tests.  A fan witnessed after assigning only one of its rows or
columns would also be present in the completion, so none of these
masks can be excluded by an initial fan witness.  The same
argument shows that any two of its row masks are compatible.  Every
row and column intersect consistently and satisfy
\eqref{tail:row-column}.  Hence, whenever a candidate from the
completion remains in a domain, the corresponding mask in every other
domain supplies a compatible support, provided the completion's other
masks have not previously been removed.  Induction on the sequence of
deletions establishes the assertion.
\end{proof}

The generator repeatedly performs these deletions until no further
deletion applies.  It then either reaches an empty domain, finds a
forbidden fan using the rows whose domains have become singletons, or
chooses a non-singleton row domain and branches over every one of its
remaining masks.  Choosing a row with the smallest non-singleton
domain is a search heuristic only; the proof relies on exhaustive
coverage of that domain.  Every branch fixes at least one previously
non-singleton row, so a path has at most eight such branch operations.
No candidate is discarded because of graph isomorphism or a lexicographic
symmetry convention in the recorded certificates.

\begin{prop}
\label{tail:certificate-sound}
If the described certificate passes the independent verifier and every
leaf is a checked contradiction, then the given seed pair has no
$(F_3,F_4)$-avoiding completion satisfying the degree bounds and
\eqref{tail:swap}.
\end{prop}

\begin{proof}
Assume such a completion exists.  By
Proposition~\ref{tail:pruning-sound}, its row and column masks survive
every sound deletion along any search path consistent with the
completion.  At each branch there is a child assigning exactly the
completion's value to the selected row, because the verifier checks
that the children contain every remaining value exactly once.  Follow
this child.  The finite tree eventually reaches a leaf.  That leaf
cannot have an empty domain, since the completion's masks survive, and
cannot contain a witnessed forbidden fan, since all edges of that fan
also occur in the completion.  Both permitted leaf types therefore
give contradictions.
\end{proof}

The independent verifier does not solve a matching problem.  To check
a claimed fan, it verifies that the matching endpoints are distinct
from one another and from the centre, that the required spokes and
matching edges are all known, and that their colours agree.  It then
replays each domain deletion using the current domains and the
explicit incompatibility witnesses or~\eqref{tail:row-column}.
Finally, it verifies exhaustive branch coverage and each leaf
contradiction.  Thus an error in the generator's matching search would
not by itself cause an unsupported exclusion to pass verification.

Table~\ref{tail:six-certificates} gives the six case identities and
certificate statistics.  The case identifier is the stable index in
the original filtered pair enumeration.  A node count of one means
that propagation alone produced a contradiction at the root, with no
branching.  The fan-witness count includes the initial pattern
rejections, the distinct row-pair rejection witnesses retained in the
certificate, and any fan contradictions at leaves.  The deletion
count records domain deletions in the replayed tree; it does not count
the initial degree exclusions.

\begin{table}[htbp]
\centering
\small
\caption{Independently checked finite-domain certificates for the six
remaining cases.  All certificates were generated without symmetry
restrictions.}
\label{tail:six-certificates}
\begin{tabular}{r r r r r r}
\toprule
Case & Left seed & Right seed & Nodes & Fan witnesses & Deletions\\
\midrule
51996400 & 3477 & 1235  & 911   & 33,134 & 62,671\\
64095209 & 4228 & 1235  & 917   & 53,482 & 39,917\\
69103064 & 4877 & 68943 & 1     & 31,625 & 1,086\\
69257548 & 4893 & 68943 & 1     & 31,766 & 1,126\\
71867158 & 5479 & 1235  & 5,668 & 33,730 & 343,586\\
82561944 & 6899 & 30887 & 1     & 14,551 & 384\\
\midrule
Total & & & 7,499 & 198,288 & 448,770\\
\bottomrule
\end{tabular}
\end{table}

The counting lemma and these six independently checked certificates
therefore exclude all forty exceptional pairs under the degree and
maximum-cut conditions of the structural reduction.

\section{Ramsey consequence and verification}
\label{sec:conclusion}

Let $\mathcal C$ denote the verified catalogue of retained ordered seed
pairs, and let $\mathcal H\subseteq\mathcal C$ be the forty exceptional
pairs excluded in Section~\ref{tail:section}. We now combine the
structural reduction and the completed computation.

\begin{thm}
\label{main:ramsey}
\[
 R(F_3,F_4)=17.
\]
\end{thm}
\begin{proof}
An admissible colouring of $K_{17}$ would have a maximum red cut
with parts of sizes eight and nine. Proposition~\ref{th17:coverage}
places its internal graphs, after relabelling, in a pair in $\mathcal C$;
the cross-edge matrix satisfies the degree and swap inequalities.
The verified bulk computation excludes every pair outside $\mathcal H$.
Lemma~\ref{tail:counting-lemma} excludes thirty-four pairs in
$\mathcal H$, and the independently checked certificates, with
soundness supplied by Proposition~\ref{tail:certificate-sound}, exclude
the other six. Hence no admissible colouring of $K_{17}$ exists, giving
$R(F_3,F_4)\le17$. Proposition~\ref{th17:lower-bound} gives the reverse
inequality.
\end{proof}

Verification of the seed catalogues, retained pair catalogue and bulk
exclusions has been completed. The bulk computation is documented by
the enumeration and execution records and the implementation validation
described in Appendix~\ref{app:reproducibility}. The forty exceptional pairs
additionally have independently checked graph-based certificates:
explicit matching witnesses verify the finite premises of the counting
argument for thirty-four pairs, and finite-domain contradiction
certificates with exhaustive branch coverage exclude the remaining six.
These complementary verification procedures establish the computational
conclusion used in Theorem~\ref{main:ramsey}.

\subsection{Discussion}
Neighbourhood matching conditions are effective at two different
scales in this problem. At the catalogue level, they give inexpensive
profiles that substantially reduce the number of seed pairs. For the
exceptional pairs, the same local restrictions yield a common
cross-edge contradiction or incompatibilities between complete
neighbourhood patterns. Explicit fan witnesses make the latter
incompatibilities easy to verify independently of the search that
finds them.

The eight- and nine-vertex parts are small enough for complete canonical
catalogues and exact subset matching tables, while their joint
restrictions constrain the seventeen-vertex colouring. This suggests
a general approach to related fan Ramsey problems: combine canonical
seed generation with a mathematically justified gluing reduction,
and use local witnesses to certify the remaining exclusions.

\subsection{Data availability}
Detailed reproduction instructions are provided in
Appendix~\ref{app:reproducibility}. The associated
source code, seed catalogues, data, and certificates are archived in
Zenodo~\citep{JiangShaoData} at
\url{https://doi.org/10.5281/zenodo.23119016}.
The source repository is available at\par\nobreak\noindent
\url{https://github.com/jianghanxin01/fan-F3F4-reproducibility}.

\subsection*{AI Disclosure}
The research ideas and overall approach were proposed by the authors.
Following the authors' instructions, OpenAI Codex was used to assist with
implementing and testing the programs. The authors carried out the main
computations on their computing cluster. Codex also assisted with checking
the computational results and with editing the manuscript.
The authors take responsibility for the mathematical arguments,
computational results, and final manuscript.

\clearpage
\appendix
\section{Reproducibility guide}
\label{app:reproducibility}
\begingroup
\setlength{\parindent}{0pt}
\setlength{\parskip}{0.35em}

\subsection{Purpose and verification scope}

This appendix gives reproduction and verification procedures for the
completed computation establishing $R(F_3,F_4)=17$. It explains how to
compile the search program, reproduce the seed catalogues and filtering
counts, run the distributed search, and verify the certificates for the
forty exceptional seed pairs. Here $F_t$ is the graph
formed by $t$ triangles with a common vertex and otherwise disjoint vertex
sets. A colouring is admissible if it contains neither a red $F_3$ nor a blue
$F_4$. The mathematical justification of the reductions is given in the
main text.

The bulk-search results and cumulative solver counts reported in the
main text were obtained by running the C++ search program on the authors'
computing cluster using 32 CPU cores. The wall-clock running time of
the large-scale search was at most 24 hours. The local certificate checks described below were
performed separately after the distributed search. Thus the cluster
statistics and the subsequent independent verification results have
separate sources.

The completed computation covers all $88\,294\,674$ retained ordered
pairs of seeds on eight and nine vertices. The bulk solver phase excludes
$88\,294\,634$ pairs. The forty pairs left unresolved by that phase are
all excluded separately: thirty-four by the counting argument in the
main text, with its finite premises checked from the seed graphs, and six
by independently replayable finite-domain certificates. Thus every
retained pair is excluded. Verification of the seed and pair catalogues
and the bulk exclusions has been completed. The quick verification below
checks the forty local exclusions; it does not rerun the bulk search or
establish catalogue completeness by itself.

The distribution provides the inputs and source needed for a fresh full
computation. It does not include the complete historical cluster
checkpoints or a checked propositional proof for each bulk exclusion.
The supplied source, the completeness arguments in the main text, and a
full rerun therefore have different roles from the local certificate
checker. File checksums establish file identity, not mathematical truth.

\subsection{Files and conventions}

Extract \file{fan_F3F4_reproducibility.zip} to a writable directory. All
commands below are issued from the extracted directory containing
\file{README.txt}, unless a different working directory is stated.
The shell examples use Linux syntax. Section~\ref{app:build} gives the
corresponding Windows build procedure.

\begin{center}
\begin{tabular}{@{}p{0.25\linewidth}p{0.68\linewidth}@{}}
\toprule
Location & Contents and purpose \\
\midrule
\file{src/} & C++17 search source, build files, and required third-party
source and licences. \\
\file{data/} & Fixed graph6 seed files, the ordered-pair catalogue, and
the dataset identity. \\
\file{tools/} & Canonical generation, catalogue preparation, and
independent data-validation programs. \\
\file{certificates/} & The forty case specifications and final
counting and finite-domain certificate data. \\
\file{verify40.py} & Entry point for independent verification of the
forty exceptional pairs. \\
\file{combine_results.py} & Checks coverage of a fresh full run and
combines it with the verified local exclusions. \\
\file{docs/} & A standalone version of this reproducibility guide in LaTeX and PDF form. \\
\bottomrule
\end{tabular}
\end{center}

A seed index and a case index are zero-based. Machine numbers are one-based.
The red graph determines the colouring; all other edges of the complete
graph are blue. A graph6 record encodes an undirected graph, using the
vertex order fixed by the corresponding seed file. The production
dataset identifier is \texttt{e38c0723b6a3be74}; the engine identifier is
\texttt{fan17-cpp-1}. The eight-vertex and nine-vertex sides are ordered:
there is no transposition identification between them.

The published files should be kept unchanged. Put new checkpoints,
timing logs, generated certificates, and validation reports in separate
working directories. These outputs are produced by the reader's run and
are not prerequisites for the archived quick check.

\subsection{Requirements}

The forty-case checker requires Python 3.9 or later and only its standard
library. The C++ search requires a C++17 compiler and threads. GCC 9 or
later, a recent Clang, or the x64 Microsoft Visual C++ toolchain can be
used. A CMake build requires CMake 3.15 or later. The included direct
build scripts provide an alternative to CMake. MiniSat and the JSON
header are supplied in the archive; no separate SAT-solver installation
is needed.

Full seed and profile regeneration requires Python 3.11 or later and
the pinned Python packages in \file{tools/requirements.txt}. Graph isomorphism is
handled by the BLISS backend exposed through igraph. Verification of
the six finite-domain certificates uses neither BLISS nor a SAT solver.
No GPU is required. Select the thread count to match the CPU resources
allocated by the cluster scheduler.

A full run writes $176\,590$ checkpoint files. Use storage with adequate
space and support for atomic file replacement. The total runtime depends
on the hardware, thread count, time limits and difficult instances; a
short pilot is not a reliable prediction of completion time.

\subsection{Quick verification of the forty exceptional pairs}
\label{app:quick}

\begin{minipage}{\linewidth}
Run:
\par\begingroup\footnotesize
\begin{verbatim}
python3 verify40.py
\end{verbatim}
\endgroup
\end{minipage}

On Windows, use \texttt{python} instead of \texttt{python3} when that is
the installed command. The program checks the fixed input identities,
the binding of case indices to the seed catalogue, the seed adjacency
data, the finite premises for the thirty-four counting exclusions, and
all six certificate trees. The required outcome is a successful process
exit and a report containing forty excluded cases, with zero unresolved
cases in this local set. The default report is
\file{results/verified40.json}; an alternative path may be selected by
\texttt{--report}.

For the six certificates the expected aggregate replay counts are
$7\,499$ search nodes, $198\,288$ explicit fan witnesses, and
$448\,770$ domain deletions. Each fan witness specifies a centre and
three or four disjoint pairs of peripheral vertices. The checker tests
the colours of all spokes and matching edges directly. It also checks
every domain deletion and every branch of the contradiction tree.
An unknown edge in a partial colouring belongs to neither colour.

The thirty-four counting exclusions and six tree certificates partition
the same fixed set of forty cases. A passed local check must not be
interpreted as a passed full-catalogue check. The full-run combination
procedure in Section~\ref{app:combine} checks this coverage separately.

\subsection{Compiling and checking the search program}
\label{app:build}

\subsubsection{CMake build}
From the archive root, run:
\par\begingroup\footnotesize
\begin{verbatim}
cmake -S src -B build -DCMAKE_BUILD_TYPE=Release
cmake --build build --config Release --parallel 4
\end{verbatim}
\endgroup
With a single-configuration Linux generator, the executable is
\file{build/fan_ramsey17}. Test it and inspect the fixed catalogue:
\par\begingroup\footnotesize
\begin{verbatim}
./build/fan_ramsey17 self-test
./build/fan_ramsey17 catalog --data data
\end{verbatim}
\endgroup
With the Visual Studio generator on Windows, use
\file{build\Release\fan_ramsey17.exe} in place of
\file{./build/fan_ramsey17}. Run the Windows commands in a developer
terminal with the x64 C++ toolchain available.

\subsubsection{Direct builds}
On Linux, the alternative command
\par\begingroup\footnotesize
\begin{verbatim}
bash src/build_linux.sh
\end{verbatim}
\endgroup
creates \file{src/fan_ramsey17}. In the remaining examples replace
\file{./build/fan_ramsey17} by \file{./src/fan_ramsey17} when using this
build.

\begin{minipage}{\linewidth}
In the x64 Native Tools Command Prompt for Visual Studio, the equivalent
Windows procedure is:
\par\begingroup\footnotesize
\begin{verbatim}
pushd src
call build_windows.bat
popd
src\fan_ramsey17.exe self-test
src\fan_ramsey17.exe catalog --data data
\end{verbatim}
\endgroup
\end{minipage}

Here the executable for subsequent commands is
\file{src\fan_ramsey17.exe}. Use ordinary single-line commands in
Windows; a backslash used for Linux line continuation is not a Windows
line-continuation character.

\subsubsection{Expected catalogue dimensions}

\begin{center}
\begin{tabular}{@{}lr@{}}
\toprule
Quantity & Expected value \\
\midrule
Eight-vertex seeds & $8\,812$ \\
Nine-vertex seeds & $115\,174$ \\
Eight-vertex profiles & $6\,860$ \\
Nine-vertex profiles & $60\,041$ \\
Retained ordered seed pairs & $88\,294\,674$ \\
Cases per full task & $500$ \\
Number of tasks & $176\,590$ \\
Cases in the last task & $174$ \\
\bottomrule
\end{tabular}
\end{center}

The executable rejects a catalogue or seed file that differs from its
fixed dataset. It also checks catalogue dimensions and basic profile
fields. These checks are useful consistency tests, not an independent
proof of exhaustive generation.

For a small executable test that does not start the full search, run:
\par\begingroup\footnotesize
\begin{verbatim}
./build/fan_ramsey17 run --data data \
  --machine 1 --machines 1 --threads 1 \
  --task-ids 0 --case-budget 5 --seconds 5 \
  --symmetry on --results results/smoke
\end{verbatim}
\endgroup
This deliberately processes at most five cases. Use a different results
directory for a new independent full computation.

\subsection{Running a full computation on one machine}

\par\begingroup\footnotesize
\begin{verbatim}
./build/fan_ramsey17 run --data data \
  --machine 1 --machines 1 --threads 4 \
  --symmetry on --seconds 5 --results results/machine01
\end{verbatim}
\endgroup

The \texttt{--seconds} option is a time limit for each case, not for
the whole job. The \texttt{--threads} option controls concurrent
workers on this machine. The symmetry option activates the symmetry
constraints described in the main text. An independent run with
\texttt{--symmetry off} is supported but may be slower. The default
task order, \texttt{spread}, is a deterministic permutation of all
assigned tasks and is not a sampling rule.

\begin{minipage}{\linewidth}
The status fields have the following meanings:
\begin{center}
\begin{tabular}{@{}p{0.16\linewidth}p{0.77\linewidth}@{}}
\toprule
Field & Interpretation \\
\midrule
\texttt{PENDING} & Cases not yet processed. \\
\texttt{UNSAT} & Cases for which the SAT solver reports that no
permitted completion exists. \\
\texttt{UNKNOWN} & Cases for which the time-limited attempt did not
decide satisfiability. These are not exclusions. \\
\texttt{SAT} & Cases with a candidate completion; the program checks
and stores its red adjacency data. \\
\bottomrule
\end{tabular}
\end{center}

\end{minipage}

Each task progress line reports counts within that task. Its
\texttt{wall\_seconds} is elapsed time since the current run began;
it is not the duration of that task alone. The final run line reports
the number of attempts in the invocation. The runner may finish
normally with unresolved cases because every assigned case has received
its time-limited attempt.

The flag \texttt{proof\_verified=false} is intentional in production
solver output: the runner does not emit and independently check a
propositional proof for every exclusion. A solver's negative result,
a repeated negative result and a replayed certificate are different
forms of evidence. This flag concerns per-instance propositional proof
checking; it does not mark an instance as unresolved or indicate that
the full computation is incomplete.

\subsection{Resuming, retrying and merging}
\label{app:resume}

Repeating a \texttt{run} command with the same results directory resumes
pending work. Already resolved cases are skipped. Previously unresolved
cases are attempted again only if \texttt{--retry-unknown} is present.
For example, after the first pass has stopped:
\par\begingroup\footnotesize
\begin{verbatim}
./build/fan_ramsey17 run --data data \
  --machine 1 --machines 1 --threads 4 \
  --symmetry on --seconds 60 --retry-unknown \
  --results results/machine01
\end{verbatim}
\endgroup
Further passes may use $600$ or $3600$ seconds per case. These are
example retry settings, not a guarantee that every case will finish.
Keep the dataset and seed ordering unchanged between passes.

After a run has stopped, create a summary and a merged checkpoint set:
\par\begingroup\footnotesize
\begin{verbatim}
./build/fan_ramsey17 merge --data data \
  --inputs results/machine01 --output results/merged_01
\end{verbatim}
\endgroup
The output directory must be new or contain no checkpoint files. The
merge writes \file{summary.json}, \file{pending_tasks.json},
\file{unknown_tasks.json}, and the merged task files. A missing task is
counted as pending; duplicate completed entries are reported. The merge
rejects inconsistent positive and negative decisions for a case.

The field \texttt{all\_pairs\_resolved} means that both the pending
and unresolved counts are zero. It alone does not mean all cases have
been excluded: the positive count must also be checked. The field
\texttt{solver\_reported\_exclusion} records whether the solver has
reported a negative decision for every case. It does not assert
independent proof verification.

To request an orderly stop, issue this command from another terminal:
\par\begingroup\footnotesize
\begin{verbatim}
./build/fan_ramsey17 stop --results results/machine01
\end{verbatim}
\endgroup
Allow the running process to finish and return to the shell before
merging or copying its checkpoints. The next explicit run clears the
stop request. Lock directories such as \file{RUN.lock} and the
individual task locks protect
against concurrent writers. If a process was killed, a lock may remain;
remove only the stale lock after confirming that its process has
terminated. Do not delete checkpoints to recover a stale lock, and do
not let two running processes write the same results directory.

\subsection{Distributed computation on six machines}

Use exactly the same code and data on every machine. Select
\texttt{--machines 6} everywhere and a distinct \texttt{--machine}
number from 1 through 6. Task $t$, numbered from zero, belongs to machine
$(t\bmod 6)+1$. Each machine writes to its own directory. For machine
1, run:
\par\begingroup\footnotesize
\begin{verbatim}
./build/fan_ramsey17 run --data data \
  --machine 1 --machines 6 --threads 4 \
  --symmetry on --seconds 5 --results results/machine01
\end{verbatim}
\endgroup
On machines 2 through 6, change both the machine number and the final
directory name accordingly, for example \texttt{--machine 2} and
\file{results/machine02}. The six workers need not start together.
The number of machines is independent of the number of threads per
machine. For scheduled cluster jobs, request the appropriate CPU
allocation and run the corresponding command inside the batch job.
Scheduler directives depend on the local cluster and are not hard-coded
into the distribution.

Retry on each machine with \texttt{--retry-unknown} and a larger time limit,
using its existing results directory and machine number. Once all
processes have stopped, collect the six directories on one machine,
preserving their distinct names, and merge:
\par\begingroup\footnotesize
\begin{verbatim}
./build/fan_ramsey17 merge --data data \
  --inputs results/machine01 results/machine02 \
           results/machine03 results/machine04 \
           results/machine05 results/machine06 \
  --output results/merged_six_01
\end{verbatim}
\endgroup
If a shared filesystem is used, the same six distinct directories can
be merged in place after the workers stop. Do not merge only the
individual summary files: case-level checkpoints are required.

To change the number of machines, first stop all writers and merge their
checkpoints. Work may then be redistributed by the new task-modulo rule.
Each new worker needs the existing checkpoint files for its assigned
tasks in its own results directory. Preserve the original merged copy
until the redistributed run has been checked. Reassigning tasks changes
the allocation of work, not the case indices or mathematical search.

\subsection{Combining a full run with the local certificates}
\label{app:combine}

After creating a complete merged checkpoint directory, run:
\par\begingroup\footnotesize
\begin{verbatim}
python3 combine_results.py --results results/merged_01 \
  --output results/combined.json
\end{verbatim}
\endgroup
Use \file{results/merged_six_01} instead for the six-machine example.
This program checks the checkpoint dataset, task ranges, status integrity
and case-level coverage, and reruns the forty-case verifier. It requires
zero pending cases and zero positive cases. Every unresolved case in
the merged run must belong to the fixed set excluded by the local
certificates. An unresolved case outside that set causes failure and
requires further computation.

The completed computation has the following decomposition:
\[
88\,294\,634\ \text{solver exclusions}
+34\ \text{counting exclusions}
+6\ \text{certificate exclusions}
=88\,294\,674.
\]
The combined report records the actual split in the reader's run; it
need not reproduce the same intermediate unresolved counts. It keeps
solver decisions and local certificate exclusions distinct. Successful
combination establishes complete case-level coverage under these checks;
it does not supply an independently checked propositional proof for
each bulk UNSAT decision.

The cumulative counts from the bulk-search phase are listed below
for comparison. The forty pairs left unresolved after that phase are
all excluded by the subsequent counting and certificate checks. These
summary counts are not checkpoint data and cannot substitute for
the case-level coverage check.
\begin{center}
\begin{tabular}{@{}lrr@{}}
\toprule
Pass & Excluded by the solver & Unresolved \\
\midrule
First & $88\,294\,434$ & $240$ \\
Second & $88\,294\,572$ & $102$ \\
Third & $88\,294\,620$ & $54$ \\
Fourth & $88\,294\,634$ & $40$ \\
\bottomrule
\end{tabular}
\end{center}
Each of these summaries reported zero pending cases and zero positive
cases. Different hardware and timing can change the intermediate
unresolved set without changing the underlying instances.

\subsection{Independent checks and witness verification}

The independent Python data validator uses the standard library:
\par\begingroup\footnotesize
\begin{verbatim}
python3 tools/search17_validate.py self-test
python3 tools/search17_validate.py data --data data \
  --profile-samples 4 --case-samples 12 \
  --output results/data_check.json
\end{verbatim}
\endgroup
It checks every seed record and the basic catalogue fields, with full
profile and indexing checks on the selected samples. The sample counts
can be increased. This command is not a second exhaustive canonical
generation and does not check all negative solver decisions.

If a run reports a positive case, preserve the stored witness and
verify it before drawing a conclusion:
\par\begingroup\footnotesize
\begin{verbatim}
python3 tools/search17_validate.py verify \
  --input results/machine01/SAT_FOUND.json --order 17 \
  --output results/witness_check.json
\end{verbatim}
\endgroup
The program directly checks the proposed colouring for forbidden fans.
A valid admissible colouring on seventeen vertices would contradict
the nonexistence result proved in the main text, so a positive result
must never be discarded
because it differs from the reported computation.

For a single seed pair, such as one of the six certified cases, use:
\par\begingroup\footnotesize
\begin{verbatim}
./build/fan_ramsey17 pair --data data \
  --left-index 3477 --right-index 1235 \
  --seconds 60 --symmetry off --output results/pair.json
\end{verbatim}
\endgroup
The pair command selects indices in the seed files, not a task number.
Its time-limited result is separate from replay of the supplied proof
certificate for that pair.

\subsection{Regenerating the seed libraries and catalogue}
\label{app:regenerate}

Carry out regeneration in a working copy. Use an isolated Python
environment, for example on Linux:
\par\begingroup\footnotesize
\begin{verbatim}
python3 -m venv .venv
. .venv/bin/activate
python -m pip install -r tools/requirements.txt
mkdir -p regenerated
\end{verbatim}
\endgroup
The requirements fix igraph 1.0.0, NumPy 2.3.5, texttable 1.7.0 and
NetworkX 3.7. Do not use
Python's optimisation flag \texttt{-O}: some preparation consistency
checks are assertions.

\subsubsection{Canonical generation through order nine}
\par\begingroup\footnotesize
\begin{verbatim}
python tools/canonical_fans.py --validate
python tools/canonical_fans.py --max-order 9 --target 17 --audit \
  --export regenerated/fan9_target17.graph6
\end{verbatim}
\endgroup
The first command tests the unrestricted generator against known small
graph counts. The second runs the restricted canonical augmentation
used in this search. No finite time limit is imposed here; the program
refuses to export an incomplete level as a full seed library.
The numbers of retained isomorphism classes at orders $0$ through $9$
are expected to be
\[
1,\ 1,\ 2,\ 4,\ 11,\ 34,\ 155,\ 941,\ 8\,812,\ 115\,174.
\]
The restrictions are hereditary fan and clique exclusions. A minimum
red degree of seven is imposed only on a completed seventeen-vertex
colouring, not on a seed during augmentation.

\subsubsection{Profiles, filters and stable case indexing}
To regenerate the eight-vertex library and construct the profile and
pair catalogue using the newly generated nine-vertex library, run:
\par\begingroup\footnotesize
\begin{verbatim}
python tools/search17_prepare.py --data regenerated/data \
  --source-nine regenerated/fan9_target17.graph6 \
  --canonical-source tools/canonical_fans.py
\end{verbatim}
\endgroup
The source-nine option is checked against the fixed nine-vertex seed
identity. As a less expensive separate test of profile preparation,
one may instead supply \file{data/fan9_target17.graph6}; this reuses the
archived nine-vertex library and does not regenerate that library.

Check the regenerated fixed inputs and the bindings of the forty
exceptional cases:
\par\begingroup\footnotesize
\begin{verbatim}
python tools/check_data.py --data regenerated/data \
  --cases certificates/cases.json
python tools/search17_validate.py data --data regenerated/data \
  --profile-samples 6 --case-samples 24 \
  --output regenerated/data_check.json
./build/fan_ramsey17 catalog --data regenerated/data
\end{verbatim}
\endgroup

The expected cumulative numbers of ordered pairs are:
\begin{center}
\begin{tabular}{@{}lr@{}}
\toprule
Stage & Ordered pairs retained \\
\midrule
All eight-by-nine seed pairs & $1\,014\,913\,288$ \\
Maximum-degree compatibility & $682\,271\,903$ \\
Cross-edge count intervals & $682\,057\,873$ \\
Joint centre-type compatibility & $165\,925\,659$ \\
Saturated maximum-degree boundary & $88\,294\,674$ \\
\bottomrule
\end{tabular}
\end{center}

The preparation writes a new catalogue, identity information and a
corresponding C++ identity header in its output tree. It also generates
profile caches and run reports for that invocation. These generated
caches are not required in the published archive. Check that the
regenerated fixed files agree with the archived identities before
using existing checkpoints. If the ordering changes, even when the
graphs are isomorphic and the counts agree, old case indices no longer
identify the same instances. Do not bypass the executable's identity
checks or mix checkpoints from different catalogues.

\subsection{Replaying and regenerating certificates}

For the ordinary verification procedure, use Section~\ref{app:quick};
certificate regeneration is optional. Each finite-domain certificate
uses the degree interval $7\leq d_R(v)\leq10$ and the necessary
one-for-one swap inequalities for a maximum red $8+9$ cut. Here
$d_R(v)$ is the red degree of vertex $v$. The checker validates these
finite constraints using the fixed seed adjacency data. The main text
proves why every admissible seventeen-vertex colouring has such a
partition.

The six certificate case indices and their seed indices are:
\begin{center}
\begin{tabular}{@{}rrrr@{}}
\toprule
Case index & Eight-vertex seed & Nine-vertex seed & Search nodes \\
\midrule
$51\,996\,400$ & $3477$ & $1235$ & $911$ \\
$64\,095\,209$ & $4228$ & $1235$ & $917$ \\
$69\,103\,064$ & $4877$ & $68943$ & $1$ \\
$69\,257\,548$ & $4893$ & $68943$ & $1$ \\
$71\,867\,158$ & $5479$ & $1235$ & $5668$ \\
$82\,561\,944$ & $6899$ & $30887$ & $1$ \\
\bottomrule
\end{tabular}
\end{center}

The CMake build in Section~\ref{app:build} also builds the certificate
generator \file{build/fan_tail_certify}. Regenerate all local proof data
in a new directory with:
\par\begingroup\footnotesize
\begin{verbatim}
python3 tools/regenerate_certificates.py \
  --binary build/fan_tail_certify \
  --output generated_certificates --seconds 120
\end{verbatim}
\endgroup
For a Visual Studio build, use
\file{build/Release/fan_tail_certify.exe} as the binary path.
The helper rebuilds the thirty-four-case counting certificate,
selects the six remaining cases from the fixed input list, calls the
C++ generator without symmetry restrictions, and independently replays
all regenerated certificates. The output directory must be empty.
The time limit is per case; if a certificate is not produced before
the limit, use a larger limit and a new output directory.

\begin{minipage}{\linewidth}
The regenerated files can also be checked in a separate invocation:
\par\begingroup\footnotesize
\begin{verbatim}
python3 verify40.py \
  --structural generated_certificates/theory_34_certificate.json \
  --csp-dir generated_certificates \
  --report results/regenerated_verified40.json
\end{verbatim}
\endgroup
\end{minipage}

Certificate generation must disable the optional twin-row symmetry
restriction, because the replay checker verifies a complete branching
tree without invoking a symmetry argument. A timeout or incomplete
tree is not an exclusion certificate. Store newly generated
certificates in a fresh directory and replay them before using them.
The helper enforces this symmetry setting. The archived certificates
are sufficient to reproduce all six independent exclusions without
rebuilding the generator.

\subsection{Online reproducibility archive}

The reproducibility archive contains the final
source code, fixed inputs, independent checkers and certificate data;
it also includes this guide so that the archive is self-contained.
Research notebooks, obsolete executables, intermediate checkpoints,
local dependency installations and development logs are not part of
that distribution. Required third-party licences are retained.

The reproducibility package is publicly archived in Zenodo at\par\nobreak\noindent
\url{https://zenodo.org/records/23119016}.
Its persistent identifier is\par\nobreak\noindent
\url{https://doi.org/10.5281/zenodo.23119016}.
The source repository is available at\par\nobreak\noindent
\url{https://github.com/jianghanxin01/fan-F3F4-reproducibility}.\par
For a fixed archived copy, use the Zenodo record. The GitHub repository
may contain later updates. When citing the computational materials,
use the Zenodo DOI, which identifies the reproducibility package
rather than this paper.

\endgroup

\end{document}